\documentclass[a4paper]{amsart}
\usepackage{amssymb}
\usepackage{graphics}
\usepackage[all]{xy}

\newtheorem{thm}{Theorem}[section]
\newtheorem{prop}[thm]{Proposition}
\newtheorem{cor}[thm]{Corollary}
\newtheorem{lem}[thm]{Lemma}

\theoremstyle{definition}
\newtheorem{dfn}[thm]{Definition}
\newtheorem{ex}[thm]{Example}

\newtheorem{que}[thm]{Question}

\theoremstyle{remark}
\newtheorem{rem}[thm]{Remark}

\newcommand{\RR}{\mathbb R}

\newcommand{\CC}{\mathbb C}

\newcommand{\rk}{\operatorname{rank}}

\newcommand{\Ad}{\operatorname{Ad}}

\newcommand{\supp}{\operatorname{supp}}

\newcommand{\mfk}{\mathfrak{k}}

\newcommand{\mfg}{\mathfrak{g}}
\newcommand{\mfh}{\mathfrak{h}}
\newcommand{\mft}{\mathfrak{t}}
\newcommand{\mfp}{\mathfrak{p}}

\newcommand{\bas}{\mathrm{bas}}
\newcommand{\Hb}{H_{\bas}}
\newcommand{\reg}{\mathrm{reg}}

\newcommand{\Mr}{M_{\reg}}

\newcommand{\Tors}{\operatorname{Tors}}
\newcommand{\depth}{\operatorname{depth}}

\newcommand{\Ext}{\operatorname{Ext}}
\newcommand{\Mmax}{M_{\max}}

\newcommand{\into}{\hookrightarrow}
\DeclareMathOperator{\tensor}{\otimes}

\renewcommand{\iff}{\Leftrightarrow}

\begin{document}

\title[Torsion and Cohen-Macaulay actions]{Torsion in equivariant cohomology and Cohen-Macaulay $G$-actions}

\author{Oliver Goertsches}
\address{Oliver Goertsches, Mathematisches Institut, Universit\"at zu K\"oln, Weyertal 86-90, 50931 K\"oln, Germany}
\email{ogoertsc@math.uni-koeln.de}
\author{S\"onke Rollenske}

\address{S\"onke Rollenske,
Mathematisches Institut, Rheinische Fried\-rich{\-}-Wil\-helms-Uni\-versi\-t\"at Bonn, 
 Endenicher Allee 60,  53115 Bonn, Germany}
\email{srollens@math.uni-bonn.de}

\begin{abstract}
We show that the well-known fact that the equivariant cohomology of a torus action is a torsion-free module if and only if the map induced by the inclusion of the fixed point set is injective generalises to actions of arbitrary compact connected Lie groups if one replaces the fixed point set by the set of points with maximal isotropy rank. This is true essentially because the action on this set is always equivariantly formal.

In case this set is empty we show that the induced action on the set of points with highest occuring isotropy rank is Cohen-Macaulay. It turns out that just as equivariant formality of an action is equivalent to equivariant formality of the action of a maximal torus, the same holds true for equivariant injectivity and the Cohen-Macaulay property. In addition, we find a topological criterion for equivariant injectivity in terms of orbit spaces.
\end{abstract}
\maketitle

\section{Introduction}

The relation between the theory of equivariant cohomology and  the topology of Lie group actions is most visible in the case of torus actions. For example,  a weak version of the famous Borel localisation theorem states that for a sufficiently nice space $M$ acted on by a real torus $T$, the natural restriction map to the fixed point set $H^*_T(M)\to H^*_T(M^T)$  is an isomorphism modulo torsion. Hence, the equivariant cohomology of a torus action encodes much information about the fixed point set and vice versa. However, this is no longer true if we replace $T$ by a general compact connected Lie group $G$: for the Borel localisation theorem to hold true one must replace  $M^T$ by $\Mmax$, the set  of points whose isotropy groups have the same rank as $G$.\footnote{This was observed for example in \cite{GGK}. See the first assertion of Theorem C.70, which is a special case of the general version of the localisation theorem as in \cite[Theorems III.1 and III.1']{Hsiang} or \cite[Theorem 3.2.6]{AlldayPuppe}.}

Our first aim is to shed some light on  a different aspect of the same story. In analogy to the torus case, again replacing $M^T$ by $\Mmax$,  we prove in Section \ref{sec:injectivity} that  the kernel of $H^*_G(M)\to H^*_G(\Mmax)$ is exactly  the torsion submodule. On the way we observe that the $G$-action on $\Mmax$ is always equivariantly formal, thereby answering a question posed by Guillemin, Ginzburg and Karshon \cite[Remark C.72]{GGK}.

The above considerations do not give any information as soon as  $\Mmax=\emptyset$, that is,  the full equivariant cohomology is a torsion module. In order to be able to study such more general actions the notion of Cohen-Macaulay action was introduced in \cite{GT}. In Section \ref{prep} we recall this definition and some of the basic properties.
If one regards the union of lowest-dimensional orbits as a natural substitute for  the fixed point set such actions share many of the good properties of equivariantly formal actions. 

In Section \ref{sec:lowrank} we generalise  the results described above to the stratum of highest isotropy rank: given a $G$-action on $M$, we show that if $b$ is the maximal rank of a $G$-isotropy, then the $G$-action on the set $M_{b,G}$ of points with $G$-isotropy rank equal to $b$ is always Cohen-Macaulay. It follows that $H^*_G(M)\to H^*_G(M_{b,G})$ is injective if and only if $H^*_T(M)\to H^*_T(M_{b,T})$ is injective.

In Section \ref{sec:characterisationofinjectivity} we investigate further the condition of equivariant injectivity of torus actions, i.e., injectivity of $H^*_T(M)\to H^*_T(B)$, where $B$ is the infinitesimal bottom stratum of the action. In many cases (e.g.~for Cohen-Macaulay actions), $B$ coincides with the set $M_{b,T}$. Combining a result from \cite{GT} with results from \cite{Duflot} we can give a geometric characterisation of equivariant injectivity in terms of the orbit space of the action. To formulate the statement we introduce the notion of having no essential basic cohomology. 


\subsection*{{Acknowledgements:}}
We are grateful to Dirk T\"oben for several enlightening discussions and to Igor Burban for some help on Cohen-Macaulay modules. The second author was partly supported by the Hausdorff Centre for Mathematics in Bonn.

\section{Preparations}\label{prep}

Starting with Section \ref{notations},  $M$ will always denote a compact differentiable manifold on which either a torus $T$ or a compact connected Lie group $G$ acts. 
However, in Sections \ref{sec:injectivity} and \ref{sec:lowrank} we encounter associated $G$-spaces that are not necessarily smooth: for example, we will investigate the action on $\Mmax=\{p\in M\mid \rk \mfg_p=\rk \mfg\}$. It is therefore essential to remark that the theory of equivariant cohomology, including central notions like equivariant formality or Cohen-Macaulayness, make sense for more general well-behaved $G$-spaces, e.g.~as in \cite[Definition 3.2.4]{AlldayPuppe}. In this section we assume (for example) that $M$ is a compact Hausdorff space such that for all connected closed subgroups $K\subset G$  we have $\dim H^*(M^K)<\infty$, where $M^K$ is the set of $K$-fixed points in $M$. Here, $H^*$ denotes \v{C}ech cohomology with real (or rational or complex) coefficients, but note that for all spaces occurring in later sections, \v{C}ech cohomology coincides with singular cohomology.

The equivariant cohomology of an action of a compact connected Lie group $G$ on $M$ is by definition the cohomology of the Borel construction
\[
H^*_G(M)=H^*(EG\times_G M).
\]
The projection $EG\times_G M\to EG/G=BG$ to the classifying space $BG$ of $G$ induces on $H^*_G(M)$ the structure of an $H^*(BG)$-algebra. Note that $H^*(BG)$ is the ring $S(\mfg^*)^G$ of $G$-invariant polynomials on the Lie algebra $\mfg$. 

\begin{dfn} An action of a compact connected Lie group $G$ on $M$ is equivariantly formal if $H^*_G(M)$ is a free module over $S(\mfg^*)^G$.
\end{dfn}
This terminology  was introduced in \cite{GKM}, and generalised in \cite{GT} to

\begin{dfn} The $G$-action on $M$ is a Cohen-Macaulay action if $H^*_G(M)$ is a Cohen-Macaulay module over $S(\mfg^*)^G$.
\end{dfn}

Our general reference for the notion and basic properties of Cohen-Macaulay modules is \cite{BrunsHerzog}. Note that the theory for modules over local rings works as well for graded modules over graded polynomial rings in several variables, see e.g.~\cite[Section 5]{GT}.

As for our spaces $H^*(M)$ is a finite-dimensional vector space, it follows from \cite[Corollary 4.2.3]{AlldayPuppe} that an action of a torus $T$ on $M$ is equivariantly formal if and only if the spectral sequence associated with the fibration $ET\times_T M\to BT$ collapses at the $E_2$-term. Traditionally, one would have said that $M$ is totally nonhomologous to zero in $ET\times_T M\to BT$. This condition in turn is the same as that $H^*_T(M)=H^*(M)\otimes S(\mft^*)$ as graded $S(\mft^*)$-modules. We can deduce the corresponding equivalence for actions of arbitrary connected compact Lie groups:
\begin{prop} \label{prop:Gequivformalfree}An action of a compact connected Lie group $G$ is equivariantly formal if and only if $H^*_G(M)=H^*(M)\otimes S(\mfg^*)^G$ as graded $S(\mfg^*)^G$-modules.
\end{prop}
\begin{proof} If the $G$-action is equivariantly formal, then $H^*_G(M)$ is by definition a free module over $S(\mfg^*)^G$. We have that $H^*_T(M)=H^*_G(M)\otimes_{S(\mfg^*)^G} S(\mft^*)$ as $S(\mft^*)$-modules by \cite[Theorem C.35]{GGK} or \cite[Proposition 1, p.~38]{Hsiang}, hence $H^*_T(M)$ is a free module over $S(\mft^*)$ and the $T$-action on $M$ is equivariantly formal. As explained above, this implies that $H^*_T(M)=H^*(M)\otimes S(\mft^*)$. But then, using that the $G$-equivariant cohomology $H^*_G(M)$ is equal to the subalgebra $H^*_T(M)^W$ of elements in $H^*_T(M)$ invariant under the Weyl group $W$ of $G$ (see again \cite[Theorem C.35]{GGK} or \cite[Proposition 1, p.~38]{Hsiang}), one sees that $H^*_G(M)=H^*(M)\otimes S(\mft^*)^W=H^*(M)\otimes S(\mfg^*)^G$ because the $W$-action on $H^*(M)\otimes S(\mft^*)$ is induced by the trivial one on $H^*(M)$ and the standard one on $S(\mft^*)$.
\end{proof}
The proof of Proposition \ref{prop:Gequivformalfree} is essentially copied from \cite[Proposition C.26]{GGK}. In particular it shows that our definition of equivariant formality coincides with the one given in \cite{GGK}, hence we have
\begin{prop}[{\cite[Proposition C.26]{GGK}}]\label{prop:GTeqformal} If $T$ is a maximal torus in $G$, then a $G$-action on $M$ is equivariantly formal if and only if the induced $T$-action is equivariantly formal.
\end{prop}

Combining Proposition \ref{prop:GTeqformal} with the fact that a $T$-action on $M$ is equivariantly formal if and only if it is Cohen-Macaulay and has fixed points \cite[Proposition 6.4]{GT} we obtain 

\begin{prop} \label{prop:CM+maxrkisEF}
A $G$-action on $M$ is equivariantly formal if and only if it is Cohen-Macaulay and there is some point with maximal isotropy rank, i.e., a point $p\in M$ with $\rk \mfg_p=\rk \mfg$.
\end{prop}

We will prove below in Proposition \ref{prop:CMforGT} a statement analogous to Proposition \ref{prop:GTeqformal} for Cohen-Macaulay actions, but before doing so we need an algebraic lemma.

Let $A$ be an $S(\mft^*)$-module. A representation of the Weyl group $W=N_G(T)/T$ of $G$ on the vector space $A$ is said to be an action on the $S(\mft^*)$-module $A$ if for all $f\in S(\mft^*), a\in A$ and $\gamma\in W$ we have $\gamma(f\cdot a)=(\gamma\cdot f)(\gamma\cdot a)$, where the $W$-action on $S(\mft^*)$ is the natural one. 
\begin{lem}\label{lem:freenessofinvmodule} If a subgroup $W'$ of the Weyl group $W$ of $G$ acts on a Cohen-Macaulay $S(\mft^*)$-module $A$, then the subspace of $W'$-invariant elements $A^{W'}$ is a Cohen-Macaulay $S(\mft^*)^{W'}$-module. If $W'=W$ and $A$ is free over $S(\mft^*)$, then $A^W$ is free over $S(\mft^*)^W$.
\end{lem}
\begin{proof}
As $S(\mft^*)$ is finitely generated over $S(\mft^*)^{W'}$, we have that $A$ is Cohen-Macaulay over $S(\mft^*)^{W'}$ of dimension $\dim T$.
The map 
\[p:A\to A^{W'}, \qquad a\mapsto \frac1 {|W'|}\sum_{\gamma \in W'} \gamma a\]
is $S(\mft^*)^{W'}$-linear and splits the inclusion $i:A^{W'}\into A$ in the category of $S(\mft^*)^{W'}$-modules. In other words, $A^{W'}$ is a direct summand of a Cohen-Macaulay module, hence itself Cohen-Macaulay (This last fact can be seen easily from the characterisation via $\Ext$-groups).

For the case $W'=W$ and $A$ being free over $S(\mft^*)$ note that $S(\mft^*)^W$ is again a polynomial ring by the Chevalley-Shephard-Todd-Bourbaki Theorem, see e.g.~\cite[Section 18-1]{Kane}, in $\dim \mft$ indeterminates. As explained in the proof of \cite[Proposition 6.4]{GT}, a graded version of the Auslander-Buchsbaum theorem, together with the fact that we have shown above that $A^W$ is Cohen-Macaulay of $\dim T$, implies that $A^W$ is free over $S(\mft^*)^W$.
\end{proof}

\begin{prop} \label{prop:CMforGT} If $T$ is a maximal torus in $G$, then a $G$-action on $M$ is Cohen-Macaulay if and only if the induced $T$-action is Cohen-Macaulay.
\end{prop}
\begin{proof}
If the $T$-action is Cohen-Macaulay, then $H^*_G(M)=H^*_T(M)^W$ is Cohen-Macaulay by Lemma \ref{lem:freenessofinvmodule}. Conversely assume that $H^*_G(M)$ is Cohen-Macaulay. We have $H^*_T(M)=H^*_G(M)\otimes_{S(\mft^*)^W} S(\mft^*)$, and because $S(\mft^*)$ is a free $S(\mft^*)^W$-module, $H^*_T(M)$ is a Cohen-Macaulay module over $S(\mft^*)^W$. But as $S(\mft^*)$ is finitely generated over $S(\mft^*)^W$, it follows from a graded version of \cite[Proposition IV.B.12]{Serre} that $H^*_T(M)$ is also Cohen-Macaulay over $S(\mft^*)$. \end{proof}

Consider now an action of a torus $T$ on $M$. Denoting by $b$ the maximal occuring dimension of a $T$-isotropy algebra, it is proven in \cite[Proposition 5.1]{FranzPuppe2003} that 
\begin{lem}\label{lem:dimension} $\dim_{S(\mft^*)}H^*_T(M)=b$.
\end{lem}

Then we have the following equivalent characterisations of Cohen-Macaulay actions:

\begin{lem}\label{lem:CMconditions}
The following conditions are equivalent for any $T$-action on $M$:
\begin{enumerate}
\item The $T$-action is Cohen-Macaulay.
\item $\depth_{S(\mft^*)} H^*_T(M)\geq b$.
\item $\depth_{S(\mft^*)} H^*_T(M)=b$.
\item $H^*_T(M)$ is a Cohen-Macaulay ring.
\end{enumerate}
\end{lem}
\begin{proof}
The equivalence of (1), (2) and (3) follows from Lemma \ref{lem:dimension} because the depth is bounded from above by the dimension.  The equivalence to (4) follows from a graded version of \cite[Proposition IV.B.12]{Serre} since  $H^*_T(M)$ is a finitely generated $S(\mft^*)$-module.
\end{proof}


\begin{rem}
In contrast to the case of equivariantly formal actions, the Leray-Serre spectral sequence of the fibration $M\to ET \times_T M\to BT$ associated to a Cohen-Macaulay action can have non-zero differentials in arbitrary high degree. For instance, consider the $T^n$-action on $S^{2n+1}$ given by 
$$
(t_0,\ldots,t_{n-1})\cdot (z_0,\ldots,z_n)=(t_0z_0,\ldots,t_{n-1}z_{n-1},z_n).
$$
This action is Cohen-Macaulay (e.g.~by \cite[Remark 6.3]{GT}, as there is a circle $S^1\subset T^n$ acting freely such that $T^n/S^1$ acts equivariantly formally on $S^{2n+1}/S^1=\CC P^n$) but not equivariantly formal as it does not have any fixed points. Therefore, the spectral sequence does not collapse at the $E_2$-term $S(\mft^*)\otimes H^*(S^{2n+1})$. Hence, the only possibly nonzero differential $d_{n+1}$ is in fact nonzero.
\end{rem}

\begin{que}
It follows easily from the characterisation of equivariant formality via the surjectivity of $H^*_T(M)\to H^*(M)$ that equivariant formality is inherited by arbitrary subtori of $T$. It is not clear to us whether this still holds true for Cohen-Macaulay actions. By \cite[Remark 6.3]{GT} the Cohen-Macaulay property passes over to subtori that contain a $b$-dimensional torus acting locally freely. From an algebraic point of view, the following proposition shows that it is true for subtori that are induced by $H^*_T(M)$-regular elements in $\mft^*$.
\end{que}

\begin{prop} Let $M$ be a compact $T$-manifold, and $x\in \mft^*\subset S(\mft^*)$ be corresponding to a codimension $1$ subtorus $T'\subset T$, i.e., $\ker x = \mft'$. If $x$ is $H^*_T(M)$-regular then there is an exact sequence
\begin{equation}\label{eq:sesregularelement}
0\to H^*_T(M)\overset{\cdot x}{\to} H^*_T(M)\to H^*_{T'}(M)\to 0
\end{equation}
of $S(\mft^*)$-modules, where $H^*_{T'}(M)$ is regarded as an $S(\mft^*)$-module via the natural map $S(\mft^*)\to S({\mft'}^*)$.
In particular the $T$-action is Cohen-Macaulay if and only if the $T'$-action is Cohen-Macaulay.
\end{prop}
\begin{proof}
As we are proving this only for compact $T$-manifolds, we may use the Cartan model (see \cite{GuilleminSternberg}) to compute equivariant cohomology.
Since $ S(\mft^*)/x S(\mft^*)=S({\mft'}^*)$ multiplication by $x$ induces an exact sequence 
\[0\to \Omega(M)^T\tensor S(\mft^*)\overset{\cdot x}{\to} \Omega(M)^T\tensor S(\mft^*)\to \Omega(M)^T\tensor S({\mft'}^*)\to 0.\]
By assumption multiplication with $x$ is injective on $H^*_T(M)$. Thus the long exact sequence in cohomology splits into short exact sequences and we are done if we can identify the cohomology of $(\Omega(M)^T\tensor S({\mft'}^*), d_{T'})$ with $H^*_{T'}(M)$.

There is an inclusion $(\Omega(M)^T\tensor S({\mft'}^*), d_{T'})\into(\Omega(M)^{T'}\tensor S({\mft'}^*), d_{T'}) $ of bigraded complexes which induces an isomorphism on the $E_1$ term of the corresponding spectral sequences \cite[p.~70]{GuilleminSternberg}
\[ \left(H^{p-q}(M)\tensor S({\mft'}^*)\right)^{T}=H^{p-q}(M)\tensor S({\mft'}^*)=\left(H^{p-q}(M)\tensor S({\mft'}^*)\right)^{T'}\]
because  $T$ acts trivially on $H^*(M)$. Since the spectral sequences converge to the cohomologies of the respective complexes we see that the cohomologies of the complexes  are isomorphic. Thus, \eqref{eq:sesregularelement} is exact.

Because $x$ is a regular element, $H^*_{T}(M)$ is Cohen-Macaulay if and only if $H^*_{T'}(M)$ is Cohen-Macaulay.
\end{proof}


\section{Some notations}\label{notations}
Consider the action of a compact connected Lie group $G$ on a compact manifold $M$. In order to simplify notation we will assume that the action is almost effective. For $p\in M$ we denote by $\mfg_p$ the Lie algebra of the isotropy group $G_p$ of $p$. Following the notation in \cite{Duflot} we set 
\begin{gather*}
 M_{i,G}:=\{p\in M\mid \rk \mfg_p\geq i\},\\
M_{(i),G}:=\{p\in M\mid \rk \mfg_p= i\}=M_{i,G}\setminus M_{i+1,G}.
\end{gather*}
In Section \ref{sec:lowrank} we will compare $M_{i,G}$ and $M_{i,T}$, where $T$ is a maximal torus in $G$, which makes it necessary to include the acting group into the notation. If we are dealing with only one action, however, we will suppress the subscript. We have inclusions
\[
M=M_0\supset M_1\supset \dots \supset M_b \supset M_{b+1}=\varnothing,
\]
where $b$ is the maximal occuring rank of an isotropy algebra.  For any Lie subalgebra $\mfk\subset \mfg$ we set 
\[ 
M^\mfk=\{p\in M\mid \mfg_p\supset\mfk\}
\]
which coincides with the fixed point set of the connected Lie subgroup $K\subset G$ with Lie algebra $\mfk$. For any point $p\in M^\mfk$ we denote by $M^{\mfk,p}$ the connected component of $M^{\mfk}$ containing $p$.  

If $G=T$ is a torus, then each $M^\mfk$ is $T$-invariant. The regular stratum of an action of a torus $T$ is by definition the set of points $p$ such that $\dim \mft_p$ is minimal among the dimensions of the occurring isotropy algebras. This open and dense subset of $M$ will be denoted by $\Mr$. Note that $(M^{\mft_p})_{\reg}=\{q\in M\mid \mft_q=\mft_p\}$, and that the manifold $M$ is stratified by the $(M^{\mft_p,p})_{\reg}$. 

The infinitesimal bottom stratum of the $T$-action is defined as the union of the closed strata, that is, of those $M^{\mft_p,p}$ for which $M^{\mft_p,p}=(M^{\mft_p,p})_{\reg}$. Clearly, it always contains the fixed point set $M^T$.

\section{The torsion module for $G$-actions}\label{sec:injectivity}

Let $G$ be a compact connected Lie group acting on a compact manifold $M$, and let $T$ be a maximal torus of $G$. By Proposition \ref{prop:GTeqformal}  the $G$-action on $M$ is equivariantly formal if and only if the $T$-action is equivariantly formal. In this section we will investigate how equivariant injectivity of the two actions correspond. 


We say that a subalgebra $\mfh\subset \mfg$ has \emph{maximal rank} if $\rk \mfh=\rk \mfg$. The set of points whose isotropy algebra has maximal rank will be denoted by $\Mmax=\{p\in M\mid \rk \mfg_p=\rk \mfg\}$. 
\begin{rem}
It is not difficult to find examples of $G$-actions for which $\Mmax$ is not a smooth submanifold of $M$.  An easy concrete example is the natural $\mathrm{SO}(3)$-action on ${\mathrm{SU}}(3)/{\mathrm{SO}(3)}$ by left translations: a look at the isotropy representation at the origin shows that $\Mmax$ is not smooth at that point.

On the other hand, $\Mmax$ is smooth if $G$ is a torus, as then $\Mmax$ coincides with the fixed point set. Another important example in which it is smooth is the $G$-action on itself by conjugation as then $G_{\max}=G$.
\end{rem}
We will show that for general $G$, $\Mmax$ in some ways behaves similarly to the fixed point set in case of a torus action. 
In \cite[Remark C.72]{GGK} the question was posed whether $H^*_G(\Mmax)$ is always a torsion-free $S(\mfg^*)^G$-module. We will answer this question positively below in Corollary \ref{cor:M_max equiv formal}.

\begin{lem} \label{lem:G/N acyclic} Let $K$ be a compact Lie group and $T$ be a maximal torus in the identity component $K^0$. Then $K/N_K(T)$ is $\RR$-acyclic, i.e.~$H^n(K/N_K(T))$ vanishes for $n>0$ and is equal to $\RR$ for $n=0$.
\end{lem}
\begin{proof} For connected $K$ this is proven in \cite[Section III.1, Lemma (1.1)]{Hsiang}. For disconnected $K$, we reduce to this case by showing that the natural map $K^0/N_{K^0}(T)\to K/N_K(T)$ is an isomorphism. Injectivity is clear as $N_{K^0}(T)=K^0\cap N_K(T)$, so we need to show that this map is surjective. If $k\in K$ is an arbitrary element, then $kTk^{-1}$ is a maximal torus in $K^0$, so there exists $g\in K^0$ with $gkT(gk)^{-1}=T$, hence $gk\in N_K(T)$. Thus $kN_K(T)=k(gk)^{-1}N_K(T)=g^{-1}N_K(T)$ is the image of $g^{-1}N_{K^0}(T)$. 
\end{proof}

For any maximal torus $T$ of $G$ we have $\Mmax=G\cdot M^T$.

\begin{prop} \label{prop:eqcohomnormalizer} For any $G$-action on $M$ and any maximal torus $T\subset G$ we have
\[
H^*_G(\Mmax)=H^*_{N_G(T)}(M^T).
\]
\end{prop}
\begin{proof} Consider the natural map $f:G\times_{N_G(T)}M^T\to \Mmax$ given by $f([g,p])=gp$. Clearly, $f$ is surjective. Note that (see e.g.~\cite[Lemma 1.1]{Hauschild}) we have $G\, p\cap M^T=N_G(T)\, p$ for every $p\in M^T$. Thus every element in $f^{-1}(p)$ is of the form $[g,p]$ for some $g\in G_p$. This means
\[ f^{-1}(p)\cong G_p/(N_G(T)\cap G_p)=G_p/N_{G_p}(T).\]
As $f$ is $G$-equivariant, we have $f^{-1}(gp)=gf^{-1}(p)$ for all $g\in G$. It follows from Lemma \ref{lem:G/N acyclic} that $H^n(f^{-1}(p))$ vanishes for all $n>0$ and is equal to $\RR$ for $n=0$. We consider the Leray spectral sequence of the induced maps on the finite-dimensional approximations of the Borel construction
\[
f_G^k:EG_k \times_{N_G(T)}M^T=EG_k \times_G (G\times_{N_G(T)} M^T)\to EG_k \times_G \Mmax.
\]
Although $f_G^k$ is not a locally trivial fibration, it is shown in \cite[Remarque 4.17.1]{Godement} that the stalks of the Leray sheaf are given by the cohomologies of the fibers because $f_G^k$ is a map between compact spaces.
Because $G$ acts freely on $EG_k$, we have that the fiber $(f_G^k)^{-1}([v,p])$ is equal to $f^{-1}(p)$ and hence acyclic. This implies that the Leray spectral sequence collapses at the $E_2$-term with 
\[
E_\infty^{rs}= E_2^{rs}=\begin{cases} H^r( EG_k \times_G \Mmax) & s=0 \\ 0 & s>0.\end{cases}
\]
For each degree  $i$ we may choose $k$ such that the $i$-th equivariant cohomologies $H^i_{N_G(T)}(M^T)$ and $H^i_G(\Mmax)$ coincide with those of the respective $k$-th finite-dimensional approximations and thus  $H^*_{N_G(T)}(M^T)=H^*_G(\Mmax)$.
\end{proof}
\begin{rem} This shows that in case of an action of a compact connected Lie group on a compact manifold $M$ all of whose isotropy groups have maximal rank, the equivariant cohomology $H^*_G(M)$ only depends on the induced action of the Weyl group of $G$ on $M^T$. Under the additional assumption of connectedness of isotropy groups, it is proven in \cite{Hauschild} that the whole $G$-action is determined by the Weyl group action on $M^T$ up to equivariant diffeomorphism. Having this fact in mind, the proposition above is not surprising as equivariant cohomology with real coefficients does not see whether isotropy groups are connected.
\end{rem}

\begin{cor}  \label{cor:M_max equiv formal} For any $G$-action on $M$, the induced action on $\Mmax$ is equivariantly formal.
\end{cor}
\begin{proof}We have
\[
H^*_G(\Mmax)=H^*_{N_G(T)}(M^T)=(H^*_T(M^T))^W=(H^*(M^T)\otimes S(\mft^*))^W,
\]
where $W$ is the Weyl group of $G$. Note that the $W$-action on $H^*(M^T)$ is not necessarily trivial. By Lemma \ref{lem:freenessofinvmodule}, $H^*_G(\Mmax)$ is a free module over $S(\mfg^*)^G=S(\mft^*)^W$.
\end{proof}

\begin{ex} One important action to which the Corollary applies is the action of $G$ on itself by conjugation. Another way to  prove that this particular action is equivariantly formal is to check directly that the restricted action of a maximal torus $T$ is equivariantly formal. We have $G^T=T$, so the $T$-action is equivariantly formal if and only if $\dim H^*(T)=\dim H^*(G)$. But it is a classical result that the real cohomology of $G$ is an exterior algebra over $\rk G$ generators, see e.g.~\cite[II.2.2.]{LieGroupsI}. For yet another proof of the equivariant formality of the action by conjugation see \cite[Section 11.9, Item 6.]{GuilleminSternberg}. 
\end{ex}

\begin{prop} A $G$-action on $M$ is equivariantly formal if and only if we have $\dim H^*(M)=\dim H^*(\Mmax)$.
\end{prop}
\begin{proof} Choose a maximal torus $T\subset G$. The $G$-action on $M$ is equivariantly formal if and only if the $T$-action on $M$ is equivariantly formal, i.e., if and only if $\dim H^*(M)=\dim H^*(M^T)$. But by Corollary \ref{cor:M_max equiv formal} the $G$-action on $\Mmax$ and hence the $T$-action on $\Mmax$ is always equivariantly formal, i.e., $\dim H^*(M^T)=\dim H^*(\Mmax)$.
\end{proof}

Fixing a maximal torus $T\subset G$, we have $M^T\subset \Mmax$. Hence, we obtain the following commutative diagram (see also the proof of Theorem C.70 in \cite{GGK}).
\[\xymatrix{
H^*_G(M)\ar[r]\ar[d]&H^*_G(\Mmax)\ar[d]\\
H^*_T(M)\ar[r]& H^*_T(\Mmax)\ar[r] & H^*_T(M^T)}
\]
In this diagram, the vertical maps are injective, and $H^*_T(\Mmax)\to H^*_T(M^T)$ is injective by Corollary \ref{cor:M_max equiv formal}. The kernel of $H^*_T(M)\to H^*_T(\Mmax)$ equals the kernel of $H^*_T(M)\to H^*_T(M^T)$, which in turn equals the torsion submodule of $H^*_T(M)$.

\begin{prop} \label{prop:kernelG} The kernel of $H^*_G(M)\to H^*_G(\Mmax)$ equals the torsion $S(\mfg^*)^G$-submodule of $H^*_G(M)$.
\end{prop}
\begin{proof} The map $H^*_G(M)\to H^*_G(\Mmax)$ is nothing but the restriction of $H^*_T(M)\to H^*_T(\Mmax)$ to $W$-invariant elements: $(H^*_T(M))^W\to (H^*_T(\Mmax))^W$. Thus, 
\begin{align*}
\ker (H^*_G(M)\to H^*_G(\Mmax))&=\ker (H^*_T(M)\to H^*_T(\Mmax)) \cap (H^*_T(M))^W\\
&=\ker (H^*_T(M)\to H^*_T(M^T)) \cap (H^*_T(M))^W\\
&=\Tors(H^*_T(M)) \cap (H^*_T(M))^W\\
&=\Tors(H^*_G(M)),
\end{align*}
where $\Tors$ denotes the respective torsion submodule.
\end{proof}

\begin{thm} For any $G$-action, the following conditions are equivalent:
\begin{enumerate}
\item $H^*_G(M)$ is torsion-free as an $S(\mfg^*)^G$-module
\item $H^*_G(M)\to H^*_G(\Mmax)$ is injective.
\item $H^*_T(M)$ is torsion-free as an $S(\mft^*)$-module
\item $H^*_T(M)\to H^*_T(\Mmax)$ is injective.
\item $H^*_T(M)\to H^*_T(M^T)$ is injective.
\end{enumerate}
\end{thm}
\begin{proof} $(3)\iff (5)$ is a standard fact, and $(1)\iff (2)$  follows from Proposition \ref{prop:kernelG}. As $H^*_T(\Mmax)\to H^*_T(M^T)$ is injective by Corollary \ref{cor:M_max equiv formal}, we have $(4)\iff (5)$. Clearly, $(3)\Rightarrow (1)$, and the argument needed for $(1)\Rightarrow (3)$ is given in \cite[p.~211]{GGK}: if $H^*_G(M)$ is torsion-free as an $S(\mfg^*)^G=S(\mft^*)^W$-module, then $H^*_T(M)=H^*_G(M)\otimes_{S(\mft^*)^W}S(\mft^*)$ is also torsion-free as an $S(\mft^*)^W$-module because $S(\mft^*)$ is a free $S(\mft^*)^W$-module. To see that $H^*_T(M)$ is torsion-free as an $S(\mft^*)$-module, assume that $f\omega=0$ for some $f\in S(\mft^*)$ with $f\neq 0$ and $\omega\in H^*_T(M)$. But then $\prod_\gamma \gamma(f)$ is a nonzero element in $S(\mft^*)^W$ that annihilates $\omega$, which is a contradiction.
\end{proof}

\section{The highest rank stratum of $G$-actions} \label{sec:lowrank}

Consider an action of a compact connected Lie group $G$ on a compact manifold $M$, fix a maximal torus $T$ of $G$, and denote by $b$ the maximal occuring dimension of an isotropy algebra of the $T$-action on $M$. Recall that by definition $M_{b,T}=\{p\in M\mid \dim \mft_p=b\}$. As every subtorus of $G$ is conjugate to a subtorus of $T$, it follows that $b$ is the maximal occuring rank of an isotropy algebra of the $G$-action.

\begin{lem} \label{lem:lowestGstratum} $G\cdot M_{b,T}=\{p\in M\mid \rk \mfg_p=b\}=M_{b,G}$.
\end{lem}

The results in the previous section are nonvoid only in the case $b=\rk G$. In this section we generalise the results given there to arbitrary $b$. 
Recall that for  a Lie subalgebra $\mfk\subset \mfg$ and a point $p\in M^\mfk$, we denote by $M^{\mfk,p}$ the connected component of $M^\mfk$ containing $p$. The components of $M_{b,T}$ are of the form $M^{\mft_p,p}$ with $p\in M_{b,T}$. Thus, Lemma \ref{lem:lowestGstratum} implies that the components of $M_{b,G}$ are of the form $G\cdot M^{\mft_p,p}$ with $p\in M_{b,T}$.

\begin{prop} \label{prop:HGHNG}
For every $p\in M_{b,T}$ we have
\[
H^*_G(G\cdot M^{\mft_p,p})=H^*_{N_G(\mft_p)}(M^{\mft_p,p}).
\]
\end{prop}
\begin{proof} We follow the lines of Proposition \ref{prop:eqcohomnormalizer}.
Consider the surjective map 
\[
f:G\times_{N_G(\mft_p)} M^{\mft_p,p}\to G\cdot M^{\mft_p,p}.
\] We first show that for $q\in M^{\mft_p,p}$ the fiber $f^{-1}(q)$ is acyclic. If $[g,q']\in f^{-1}(q)$, then $gq'=q$. The compact Lie algebra $\mfg_q=\Ad_g \mfg_{q'}$ contains the two maximal abelian subalgebras $\mft_p$ and $\Ad_g \mft_p$, hence there is $h\in G_p$ such that $\Ad_{hg}\mft_p=\mft_p$, i.e., $hg\in N_G(\mft_p)$. Then $[g,q']=[h^{-1},hgq]=[h^{-1},p]$, whereupon $f^{-1}(p)=G_p/N_G(\mft_p)\cap G_p=G_p/N_{G_p}(\mft_p)$ which is acyclic by Lemma \ref{lem:G/N acyclic}.

Then the induced map on the Borel construction 
\[
f_G:EG\times_{N_G(\mft_p)}M^{\mft_p,p}=EG\times_G(G\times_{N_G(\mft_p)}M^{\mft_p,p})\to EG\times_G G\cdot M^{\mft_p,p}
\]
induces the desired isomorphism in cohomology.
\end{proof}
\begin{cor} \label{cor:MbGCM}For any $G$-action on $M$ the induced action on $M_{b,G}$ is Cohen-Macaulay.
\end{cor}
\begin{proof}
Write $M_{b,G}$ as the disjoint union $\bigcup_i G\cdot M^{\mft_{p_i},p_i}$ for some points $p_i\in M_{b,T}$. Denoting $\Gamma_i =N_{N_G(\mft_{p_i})}(T) $ we have
\[
H^*_G(M_{b,G})=\bigoplus_i H^*_{N_G(\mft_{p_i})}(M^{\mft_{p_i},p_i}) = \bigoplus_i H^*_{\Gamma_i}(M^{\mft_{p_i},p_i})
\]
where the second equality is due to \cite[Proposition III.1(i)]{Hsiang}, which is valid for arbitrary (also disconnected) compact groups because of Lemma \ref{lem:G/N acyclic}. We have
\[
T\subset \Gamma_i\subset N_G(T),
\]
hence $W_i:=\Gamma_i/T$ is a subgroup of the Weyl group $N_G(T)/T$ of $G$. Hence
\[
H^*_G(M_{b, G})=\bigoplus_i (H^*_T(M^{\mft_{p_i},p_i}))^{W_i}= \bigoplus_i (H^*(M^{\mft_{p_i},p_i}/T)\otimes S(\mft_{p_i}^*))^{W_i}.
\]
Lemma \ref{lem:freenessofinvmodule} shows that this module is Cohen-Macaulay. 
\end{proof}

\begin{rem}
If for some $i>b$ there exists a component $N$ of $M_{i,G}$ that does not intersect $M_{i+1,G}$, then statements analogous to Lemma \ref{lem:lowestGstratum},  Proposition \ref{prop:HGHNG} and Corollary \ref{cor:MbGCM} hold for $N$: we have $N=G\cdot M^{\mft_p,p}$ for some point $p\in M_{i,T}$ and the $G$-action on $N$ is Cohen-Macaulay. 
\end{rem}

We have the following commutative diagram:
\[\xymatrix{
H^*_G(M)\ar[r]\ar[d]&H^*_G(M_{b, G})\ar[d]\\
H^*_T(M)\ar[r]& H^*_T(M_{b, G})\ar[r] & H^*_T(M_{b, T})}
\]
In this diagram the vertical maps are just the inclusions of the respective subrings of Weyl-invariant elements. As the $G$-action on $M_{b, G}$ is Cohen-Macaulay, Proposition \ref{prop:CMforGT} implies that also the $T$-action on $M_{b, G}$ is, so $H^*_T(M_{b, G})\to H^*_T(M_{b, T})$ is injective (see e.g.~the Atiyah-Bredon sequence as in \cite[Theorem 6.2]{GT} or \cite[Theorem 5.2]{FranzPuppe2003}).

\begin{thm} The following conditions are equivalent:
\begin{enumerate}
\item $H^*_G(M)\to H^*_G(M_{b, G})$ is injective.
\item $H^*_T(M)\to H^*_T(M_{b, G})$ is injective.
\item $H^*_T(M)\to H^*_T(M_{b,T})$ is injective.
\end{enumerate}
\end{thm}
\begin{proof} $(2)\Leftrightarrow (3)$ follows because $H^*_T(M_{b, G})\to H^*_T(M_{b, T})$ is injective and $(2)\Rightarrow (1)$ is clear because $H^*_G(M)$ and $H^*_G(M_{b, G})$ are just the modules of Weyl-invariant elements in the corresponding $T$-equivariant cohomologies. Given $(1)$, the map $H^*_T(M)\to H^*_T(M_{b, G})$ is injective as it can be written as
\[
H^*_T(M)= H^*_G(M)\otimes_{S(\mft^*)^W}S(\mft^*)\to H^*_G(M_{b, G})\otimes_{S(\mft^*)^W}S(\mft^*)= H^*_T(M_{b, G})
\]
and $S(\mft^*)$ is a free $S(\mft^*)^W$-module.
\end{proof}

\begin{que} It would be interesting to know whether there is an Atiyah-Bredon sequence for general $G$-actions whose exactness is equivalent to equivariant formality respectively Cohen-Macaulayness, see \cite{FranzPuppe, GT}. This sequence would involve the expressions $H^*_G(M_{i,G},M_{i+1,G})$, where $M_{i,G}=\{p\in M\mid \rk \mfg_p\geq i\}$. Our results stating that although $H^*_G(\Mmax)$ respectively $H^*_G(M_{b,G})$ are not explicitly calculable, they are still free respectively Cohen-Macaulay modules, support this conjecture. Is $H^*_G(M_{i,G},M_{i+1,G})$ always Cohen-Macaulay?
\end{que}

\section{A characterisation of equivariant injectivity}\label{sec:characterisationofinjectivity}

Consider an action of a torus $T$ on a compact manifold $M$. As before we will denote by $b$ the maximal occuring dimension of a $T$-isotropy algebra.  As there is no danger of confusion we will write $M_i$ for $M_{i,T}$. The goal of this section is to prove an equivalent characterisation of injectivity of the map $H^*_T(M)\to H^*_T(B)$, where $B$ denotes the infinitesimal bottom stratum of the $T$-action. Note that in many cases, e.g.~for Cohen-Macaulay actions, $B$ equals $M_{b,T}$. If the $T$-action is Cohen-Macaulay, then $H^*_T(M)\to H^*_T(B)$ is injective, see \cite[Theorem 6.2]{GT} or \cite[Theorem 5.2]{FranzPuppe2003}.

First we will recall some useful results obtained by Duflot: using the notations from Section \ref{notations},  $M_{(i)}$ is a closed submanifold of $M\setminus M_{i+i}$ \cite[Proposition 6]{Duflot}, and we obtain a push-forward map $H^*_T(M_{(i)})\to H^*_T(M\setminus M_{i+1})$ which raises degree by the codimension of $M_{(i)}$.
Duflot proved \cite[proof of Theorem 1]{Duflot} that we have short exact sequences
\begin{equation} \label{eqn:sesDuflot}
0\longrightarrow H^*_T(M_{(i)})\longrightarrow H^*_T(M\setminus M_{i+1})\longrightarrow H^*_T(M\setminus M_i)\longrightarrow 0.
\end{equation}

\begin{rem}\label{abcde} It is not hard to show  via these short exact sequences (the necessary arguments are implicit in the proof of  \cite[ Theorem 3.1]{Duflot2}) that every associated prime ideal of $H^*_T(M)$ is of the form $\mfp_\mfk=\{f\in S(\mft^*)\mid \left.f\right|_\mfk=0\}$ for some isotropy algebra $\mfk$. In particular, if $\mfk_1, \dots, \mfk_r$ are the maximal isotropy algebras (in the partial order given by inclusion), then the corresponding prime ideals  $\mfp_1, \dots, \mfp_r$ are exactly the minimal associated primes of $H^*_T(M)$. Thus $\supp H^*_T(M)=V(\mfp_1, \dots, \mfp_r)=\bigcup \mfk_i$. This also gives a different proof of Lemma \ref{lem:dimension} for compact $T$-manifolds.
\end{rem}

As in \cite{Duflot} we define 
\[
F_i=\ker H^*_T(M)\to H^*_T(M\setminus M_i)
\] 
obtaining a filtration
\[
0=F_{b+1}\subset F_b\subset\dots \subset F_{1}\subset F_0=H^*_T(M).
\]
We will make use of \cite[Theorem 1.(1)]{Duflot}: 
\begin{equation}\label{eqn:FiQuotient}
F_i/F_{i+1}\cong H^*_T(M_{(i)}).
\end{equation}
In \cite{Duflot} it is only claimed that this is an isomorphism of vector spaces, but it is an isomorphism of $S(\mft^*)$-modules as well, as can be seen from the way the isomorphism is constructed: in the diagram on the bottom of p.~260 in \cite{Duflot}, all maps respect the $S(\mft^*)$-module structure.

The following lemma is well-known.

\begin{lem}[{\cite[Proposition 1.2.9]{BrunsHerzog}}] \label{lem:depthses} Let $0\to A\to B\to C\to 0$ be an exact sequence of finitely generated graded $S(\mft^*)$-modules. Then the following hold:
\begin{enumerate}
\item $\depth A\geq \min \{\depth B,\depth C+1\}$
\item $\depth B\geq \min\{\depth A,\depth C\}$
\item $\depth C\geq \min \{\depth A-1,\depth B\}$
\end{enumerate}
\end{lem}

\begin{lem} \label{lem:depthFi}
We have $\depth F_i \geq i$.
\end{lem}
\begin{proof}  

By \eqref{eqn:FiQuotient} we have for all $i$ an exact sequence
\[0\to F_{i+1}\to F_i\to H^*_T(M_{(i)})\to 0. \]  Let $Y_{i,j}$ be the connected components of $M_{(i)}$, and denote the ($i$-dimensional) isotropy algebra of $Y_{i,j}$ by $\mft_{i,j}$. Then we have $H^*_T(M_{(i)})=\bigoplus H^*(Y_{i,j}/T)\otimes S(\mft_{i,j}^*)$ which implies 
\[\depth H^*_T(M_{(i)})=i.\] As $F_{b+1}=0$, this implies $\depth F_b =b$. By induction, $F_i$ sits in an exact sequence where the term on the left has depth $\geq i+1$ and the one on the right has depth $i$. By Lemma \ref{lem:depthses} this implies $\depth F_i \geq i$.
\end{proof}

We first adress the question when the equivariant cohomology $H^*_T(M)$ has depth $0$ over $S(\mft^*)$. The following definition turns out to be convenient in the statement of our results. Recall that the basic cohomology $\Hb^*(N)$ of a $T$-manifold $N$ is the cohomology of the complex of horizontal and invariant differential forms on $N$. By \cite[Theorem 30.36]{Michor} it coincides with the singular cohomology of the orbit space $N/T$. We have a natural map $\Hb^*(N)\to H^*_T(N)$, which in general is neither injective \cite[Example C.18]{GGK} nor surjective. For locally free actions, however, it is an isomorphism.
\begin{dfn}
Let $U$ be a small equivariant open neighbourhood of $M_{(1)}$ in $M\setminus M_{2}$. We say the $T$-action on $M$ \emph{has no essential basic cohomology} if 
\begin{equation}\label{eqn:essbasiccohom}
\Hb^*(\Mr)\to \Hb^*(U\cap \Mr)
\end{equation}
 is injective.
\end{dfn}

Informally speaking, the next proposition shows that $H^*_T(M)$ has depth 0 if and only if the orbit space of the regular stratum has some topology that is not detected by $M^{(1)}/T$.

\begin{prop} \label{depthesscohom}
If the $T$-action on $M$ is almost effective, then $\depth H^*_T(M)\geq 1$ if and only if the $T$-action has no essential basic cohomology.
\end{prop}

\begin{proof}
First of all, consider the exact sequence
\[
0 \longrightarrow F_{2} \longrightarrow H^*_T(M) \longrightarrow H^*_T(M\setminus M_{2}) \longrightarrow 0.
\]
Because $\depth F_{2}\geq 2$ by Lemma \ref{lem:depthFi}, Lemma \ref{lem:depthses} implies that $\depth H^*_T(M)=0$ if and only if  $\depth H^*_T(M\setminus M_{2})=0$.

Let $i:M_{(1)}\into M\setminus M_{2}$ be the inclusion. We have the commutative diagram 
\[\xymatrix{ 0\ar[r] &H^*_T(M_{(1)})\ar[r]\ar[dr]_{\cdot e} & H^*_T( M\setminus M_{2})\ar[d]^{i^*} \ar[r]^{j^*} & H^*_T(\Mr)\ar[r]&0\\
 && H^*_T(M_{(1)})}
\]
in which the top row is short exact by \eqref{eqn:sesDuflot}, and $e$ is the Euler class of the normal bundle of $M_{(1)}$ in $M\setminus M_2$. Because multiplication with $e$ is injective by \cite[Proposition 4]{Duflot}, 
we see that $(i^*, j^*)$ is injective. Noting that $H^*_T(U)=H^*_T(M_{(1)})$, we have that the Mayer-Vietoris sequence of the covering $M\setminus M_{2}=U\cup \Mr$ becomes a short exact sequence
\[
0  \longrightarrow H^*_T(M\setminus M_{2})\overset{(i^*, j^*)}{\longrightarrow}   H^*_T(M_{(1)}) \oplus H^*_T(\Mr)\longrightarrow H^*_T(U\cap \Mr)\longrightarrow 0.
\]

In particular we see that $\ker(i^*)=0$ if and only if the $T$-action has no essential basic cohomology. It remains to show that $i^*$ is injective if and only of $\depth H^*_T(M\setminus M_2)\neq 0$.  If $i^*$ is injective, we obtain a short exact sequence
\[
0\longrightarrow H^*_T(M\setminus M_2)\longrightarrow H^*_T(M_{(1)})\longrightarrow H^*_T(M\setminus M_2,M_{(1)}) \longrightarrow 0
\]
Because $\depth H^*_T(M_{(1)})=1$, Lemma \ref{lem:depthses} implies that $\depth H^*_T(M\setminus M_{2})\neq 0$. If $i^*$ is not injective, then there is an element in $H^*_T(M\setminus M_2)$ coming from   
\[H^*_T(M\setminus M_{2}, M_{(1)})=H^*_T(M\setminus M_{2}, U)=H^*_T(\Mr, U\setminus M_{(1)}).\]
As the action of $T$ on  $\Mr$ is locally free this element has support $0$. Consequently, $\depth H^*_T(M\setminus M_2)=0$.
\end{proof}

Note that if the $T$-action has no essential basic cohomology, then 
\[
H^*_{\bas, c}(\Mr)\to \Hb^*(\Mr)
\]
is the zero map, where $H^*_{\bas,c}$ denotes basic cohomology with compact support. This follows because this map sits in the long exact sequence of the pair $(\Mr,V)$, where $V$ is a tubular neighborhood of $M\setminus \Mr$, and the inclusion $U\to \Mr$ factors through $V$. In general, this is not an equivalent characterisation of having no essential basic cohomology, as the following example shows.

\begin{ex} \label{ex:cylinderasorbitspace} Consider the $T^2$-action on $\CC P^2$ given in homogeneous coordinates by 
\[
(t_0,t_1)\cdot [z_0:z_1:z_2] = [t_0z_0,t_1z_1,z_2].
\] 
This equivariantly formal action has a triangle as orbit space, with the three fixed points as vertices. Let $D^2$ be a disk in the regular orbit space and remove $D^2\times T^2$ from $\CC P^2$. Denote by  $M$ the manifold obtained by equivariantly attaching two copies of $\CC P^2\setminus D^2\times T^2$ along the boundary in the natural way, see Figure \ref{fig:orbitspace}.

\begin{figure}[h]
\includegraphics{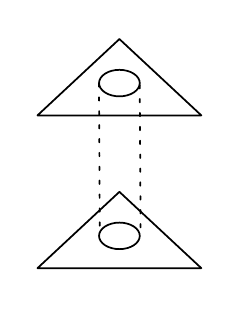}
\caption{Orbit space of Example \ref{ex:cylinderasorbitspace}}\label{fig:orbitspace}
\end{figure}
Then $M$ inherits a $T^2$-action, and the regular orbit space of that action is an open cylinder. Although for this action $H^*_{bas, c}(\Mr)\to \Hb^*(\Mr)$ is the zero map, the $T$-action has essential basic cohomology.
\end{ex}
However, if $\dim T=1$ we have
\begin{cor} For an almost effective $S^1$-action on a compact manifold $M$ the following conditions are equivalent:
\begin{enumerate}
\item The action is equivariantly formal 
\item The action has no essential basic cohomology 
\item $H^*_{\bas,c}(\Mr)\to \Hb^*(\Mr)$ is the zero map.
\end{enumerate}
\end{cor}
\begin{proof} For an $S^1$-action, $M_2=\emptyset$ and $M^1=M^{S^1}$. Equivariant formality is equivalent to $\depth H^*_T(M)\geq 1$, and thus the equivalence of $(1)$ and $(2)$ is Proposition \ref{depthesscohom}. The equivalence of $(2)$ and $(3)$ is clear.
\end{proof}
By combining a result in \cite{GT} with Proposition \ref{depthesscohom}, the notion of having essential basic cohomology allows to give an equivalent characterisation of equivariant injectivity. See Section \ref{notations} for the definition of the infinitesimal bottom stratum.
\begin{thm} \label{thm:charactinjectivity} Let $B$ denote the infinitesimal bottom stratum of the $T$-action on $M$. Then the following conditions are equivalent:
\begin{enumerate}
\item $H^*_T(M)\to H^*_T(B)$ is injective
\item for all $p\notin B$, we have $\depth H^*_T(M^{\mft_p,p})\geq \dim \mft_p+1$
\item for all $p\notin B$, the $T$-action on $M^{\mft_p,p}$ has no essential basic cohomology.
\end{enumerate}
\end{thm}
\begin{proof} 
For $p\in M$ choose a subtorus $T'\subset T$ such that $\mft=\mft_p\oplus \mft'$. Then $H^*_T(M^{\mft_p,p})=H^*_{T'}(M^{\mft_p,p})\otimes S(\mft_p^*)$, hence $\depth H^*_T(M^{\mft_p,p})=\depth H^*_{T'}(M^{\mft_p,p})+\dim \mft_p$. Because the $T'$-action is almost effective, $(2)\Leftrightarrow (3)$ follows from Proposition \ref{depthesscohom}.

Proposition 9.1 of \cite{GT} shows that $H^*_T(M)\to H^*_T(B)$ is injective if and only if for each $p\notin B$, $H^*_T(M^{\mft_p,p})$ does not contain an invisible element, i.e., an element whose support does not contain any nonregular isotropy algebra of the $T$-action on $M^{\mft_p,p}$. This in turn holds if and only if $H^*_{T'}(M^{\mft_p,p})$ does not contain an invisible element, which is equivalent to 
\[
1\leq \depth H^*_{T'}(M^{\mft_p,p})=\depth H^*_T(M^{\mft_p,p})-\dim \mft_p.
\]
This implies $(1)\Leftrightarrow (2)$.
\end{proof}

\begin{ex} Theorem \ref{thm:charactinjectivity} shows that the equivariant cohomology ot the $T^2$-action constructed in Example \ref{ex:cylinderasorbitspace} is not torsion-free. In fact, in this example $B=M^T$, hence injectivity of $H^*_T(M)\to H^*_T(B)$ is equivalent to $H^*_T(M)$ being torsion-free, but as noted in Example \ref{ex:cylinderasorbitspace}, the action has essential basic cohomology.
\end{ex}

\end{document}